\documentclass{amsart}
\usepackage{hyperref}

\usepackage[ruled, lined, linesnumbered]{algorithm2e}
\SetKwProg{Try}{try}{:}{}
\SetKwProg{Catch}{catch}{:}{end}
\usepackage{graphicx}
\usepackage{extarrows}
\usepackage{amssymb,amscd}
\usepackage[all]{xy}
\usepackage{longtable}
\usepackage[latin1]{inputenc}
\usepackage{tikz,arydshln}
\usetikzlibrary{shapes,arrows}

\newtheorem{theorem}{Theorem}[section]
\newtheorem{lemma}[theorem]{Lemma}%
\newtheorem{example}[theorem]{Example}%
\newtheorem{corollary}[theorem]{Corollary}%
\newtheorem{remark}[theorem]{Remark}

\newcommand{\rmv}[1]{}

\newcommand{\Z}{{\mathbb Z}}

\newcommand{\F}{{\mathbb F}}

\newcommand{\Q}{{\mathbb Q}}

\newcommand{\fa}{\mathfrak{a}}

\DeclareMathOperator{\MPI}{\mathrm{MPI}}
\newcommand{\Int}{\mathit{Int}}
\newcommand{\Poly}{\mathit{Poly}}
\newcommand{\cal}{\mathcal}

\begin{document}
\title
{Deterministic and Efficient Ideal Arithmetic via
  Two-Element Representations}
\footnote{The AI disclosure: The main mathematical results
  were obtained by the human author. 
The first two versions (found in axviv.org:2606.26993)
were written with little
involvement of AI (except for spelling check). In the third version,
AI (Claude Fable 5) was used to simplify Section~\ref{sec:dedekind},
to fill in the details of proofs of Lemma~\ref{lem:polygcd} and
Theorem~\ref{thm:main}, and to conduct the experimental work.
}

\author{Qi Cheng}
\address{
   School of Computer Science\\
   University of Oklahoma\\
 Norman, OK 73019, USA
}
\email{qchengnorman@outlook.com}
\maketitle

\begin{abstract}
  Given an ideal in a number field, it is desirable in many situations
  to find two elements that generate the ideal over the ring of the
  integers of the field. We present a deterministic
  polynomial time algorithm to find the two-element representation of
  an ideal.  For a monic irreducible integral polynomial \( f(x) \),
  let \( K=\Q[x]/(f) \) be the number field, and
  \( O_K \) be the integral closure. Our algorithm
  handles all ideals for monogenic \( f(x) \),
  a class that includes the cyclotomic polynomials widely used in
  lattice based cryptography.  More generally, it works when the
  integer generator of the input ideal is co-prime to the index
  \( [O_K:\Z[x]/(f)] \).  A key technical ingredient is a generalized
  version of Dedekind's criterion. We use it to construct normal
  presentations of the intermediate factors that share an integer
  generator, even when its prime factorization is unknown.
\end{abstract}

\section{Introduction}

Due to the imminent threat of quantum algorithms to
cryptosystems based on prime factorization and discrete logarithm
\cite{Shor94}, cryptographers are motivated
to design new cryptographic schemes that can resist quantum attacks.
Hard ideal lattice problems underpin the security of the Ring
Learning-With-Errors (RLWE) problem, a versatile and widely used
primitive in post quantum cryptography \cite{HPS98,LPR10,SSTX09}.
A large number of algorithms are proposed, including NewHope
\cite{NewHope}, Crystals-Kyber \cite{Kyber} and LAC \cite{LAC}.
These applications call for efficient implementation of
ideal arithmetic, an important topic in computational number
theory that was first studied in \cite{Cohen93,Cohen2012},
and the references therein.
It is well-known that for a number field,
two elements are sufficient to generate any fractional ideal as a
module over its ring of integers. Known methods for finding such a
representation include random sampling and deterministic ideal
refinement.
An ideal that admits a single generator is called principal. 
Interestingly there is no known efficient algorithm
to decide the principality of an ideal, let alone to find a
generator, even if randomness is allowed.

\subsection{Description of the problem}

Let \( f(x) \) be a monic irreducible polynomial over \( \Z \)
of degree \( n \). Let \( K= \Q[x]/(f) \) be the number field.
We denote by \( O_K \)  the integral closure
of \( \Z[x]/(f) \). It is a free \( \Z \)-module, but finding
its integral basis is in general difficult: it is polynomial time
equivalent to finding the largest square factor of an integer
\cite{Chistov89,BL94}.
An ideal in \( O_K \) is also a free \( \Z \)-module.
One \cite{Cohen2012} can represent an ideal by its \( \Z \)-basis,
each of its elements being written in a \( \Z \)-basis of \( O_K \).
This gives us an integral matrix (lattice) representation for an
ideal.
A two-element representation saves not only space but also time:
for suitable (normal) representations, multiplying two ideals
costs essentially one multiplication of elements of \( O_K \)
\cite{FHHJ17}, instead of operations on \( n \times n \) integral
matrices.

By sending \( x \) to one of the roots \( \alpha \)
of \( f(x) \) in \( K \), we establish an exact sequence
of \( \Z \)-modules: 
\[ 0 \rightarrow \langle f(x) \rangle \rightarrow  \Z[x] \rightarrow O_K  \]
If the index \( [ O_K:\Z[x]/(f) ] \) is 1, then the field---
and the polynomial \( f(x) \)---
are called monogenic; the ring of integers has a power basis
\( \{1, \alpha, \alpha^2, \cdots, \alpha^{n-1}\} \) \cite{Gaal19};
and the sequence can be extended to \( 0 \):
\[ 0 \rightarrow \langle f(x) \rangle \rightarrow  \Z[x] \rightarrow
  O_K \rightarrow 0. \]
This short exact sequence allows us to lift an ideal in \( O_K \) 
back to  an ideal in \( \Z[x]  \), and  compute in \( \Z[x] \),
which is often more efficient than  dealing with integral
matrix representations of ideals.
Moreover, ideals  \( I \subseteq \Z[x] \) and
\( I + f(x)\Z[x] \subseteq \Z[x]\)  
will be mapped to the same ideal in \( O_K \).
It means that we may freely add \( f(x) \) into an
ideal in \( \Z[x] \), without changing its image in \( O_K \).
Note that \( \Z[x] \) has ideals that can not be generated by
two elements (for instance \( \langle 5, x \rangle^3 \)).
Throughout the paper, \( \langle g_1, g_2, \cdots \rangle \) denotes
the ideal generated by \( g_1, g_2, \cdots \).  It is an ideal of
\( \Z[x] \) when the generators are written as polynomials in \( x \),
and an ideal of \( O_K \) when they are written as elements of
\( O_K \), such as \( a(\alpha) \); in particular
\( \langle M, c(\alpha) \rangle = M O_K + c(\alpha) O_K \).

In general, the problem of finding a two-element representation
can be stated as follows: given
\( a_1, a_2, \cdots, a_r  \) in \( O_K \),
find \( b_1, b_2 \) in \( O_K \) so that
\[ a_1 O_K + a_2 O_K + \cdots + a_r O_K = b_1 O_K + b_2 O_K  \]
For example, if the ideal is given by its \( \Z \)-basis,
then \( r=n \) and we let \( a_1, a_2, \cdots, a_r \) be the basis.
We can assume that \( r=3 \), since \( a_4, a_5, \cdots \) 
can be absorbed into two elements recursively.
We can find an integer in the input ideal and assign it to
\( a_1 \). One of the two generators of the output ideal
can also be an integer.
So we should go one step further, and assume that
both \( a_1 \) and \( b_1\) are integers, and reformulate
the problem as:

\begin{quote}
 {\bf Input:} \( {\cal N} \in \Z, a(x) \in \Z[x], b(x)  \in \Z[x]\)

 {\bf Output:} \( M \in \Z, c (x) \in \Z[x] \) so that
 \[ {\cal N} O_K + a (\alpha) O_K + b (\alpha) O_K
   = M O_K + c(\alpha) O_K  \]
\end{quote}

In this paper, we focus on the most difficult case
where \( {\cal N} \) is a large composite number.
Once we have an efficient algorithm for the above problem,
ideal addition is straight-forward:
\[ \langle N_1, a(\alpha) \rangle + \langle N_2, b(\alpha) \rangle = \langle \gcd(N_1, N_2), a(\alpha), b(\alpha) \rangle,  \]
where the RHS can then be transformed into two-element representation.
Ideal multiplication and division are reduced to the same problem;
see Corollary~\ref{cor:arith}.
Special cases of ideal multiplication are core subroutines in
our algorithm, and they will be treated in
Lemma~\ref{lem:relativeprime} and
Section~\ref{sec:powerandmulti}.

\subsection{Previous works}

It has long been observed that two-element representations
can speed up computation significantly.
A simple randomized algorithm was proposed
and analyzed in \cite{MMPT23,PS21,FS10}.
The algorithm first finds \( a\in I \),
and then picks a random element \( r \in I/(a) \).
It can be proved that with high probability,
\( I = a O_K + r O_K \).
Pohst and Zassenhaus \cite[Section 6.3]{PZ97}
introduced normal presentations
\( \langle a,\alpha \rangle \) of an ideal. For two presentations
normal with respect to the same set of primes, the product is
\( \langle a,\alpha \rangle\langle b,\beta \rangle = \langle ab,\alpha\beta \rangle \); see also
\cite[Theorem 6.2]{FHHJ17}. The integer \( a \) may be
composite and need
not be factored. One method for obtaining a normal presentation
samples an element of the ideal and tests it
\cite[Theorem 6.5]{FHHJ17}. Fieker, Hofmann and Sircana \cite{FHS20}
analyze the Euclidean algorithm over \( \Z/n \) in the presence of zero
divisors and use it to compute the norm and the
minimum of an ideal that
is already given in normal presentation. 
Other related work can be found in \cite{Bernstein96}.

A deterministic approach can be applied when
ideals have small norms \cite{PZ85,GMN13}. 
It is difficult to generalize their method to ideals of
large norm, due to a few apparent difficulties here.
First the norm of the ideal needs to be factored into primes,
which is hard. Second one must find  prime ideals
lying above a rational prime, which  requires factoring
polynomials over finite fields. It relies on  randomized algorithms
in the current state of the art \cite{KedlayaUm11}.
Belabas \cite[Subsection 6.3.2]{Belabas04} sketches a deterministic
polynomial-time algorithm based on coprime ideal refinement and
deems it impractical when the ideal cannot be factored, because it
requires many ideal multiplications using \( \Z \)-generators.
A second deterministic polynomial-time construction follows
from the cyclic-module theorem of Lenstra and
Silverberg\cite[Theorem~1.1]{LS16}.
Given an integer \( {\cal N}>1 \) in \( {\cal I} \), apply their algorithm
to the ideal \( {\cal I}/{\cal N}^2 O_K \) of the finite principal ideal
ring \( O_K/{\cal N}^2 O_K \). A lift \( c \) of the resulting generator
satisfies \( {\cal I}=\langle {\cal N}^2,c\rangle \). Since
\( {\cal N}\in{\cal I} \), this forces \( v_Q(c)=v_Q({\cal I}) \) at
every prime ideal \( Q \) dividing \( {\cal N} O_K \), so
\( \langle {\cal N},c\rangle \) is a normal presentation.
Our algorithm works directly with the three input generators
\( {\cal N}, a(\alpha), b(\alpha) \) in a power basis and gives
the explicit bit bound in Theorem~\ref{thm:main}.

\subsection{Our result}

Assume that the input ideal is \( \langle {\cal N}, a(\alpha), b(\alpha) \rangle \).
Our algorithm can be divided into three stages.

{\bf Stage 1}
In the first stage of our algorithm, we factor the input ideal
into a product of ideals in two-element representation,
and ideals of a special form, described below.
The basic idea is to try to compute \( \gcd(a(x), b(x)) \) 
over \( \Z/{\cal N} \) using the Euclidean algorithm. 
If the computation succeeds, we find a polynomial \( c (x) \),
and obtain two-element representation
\( {\cal N} O_K + c(\alpha) O_K \).
This should solve the problem for many inputs if \( {\cal N} \)
is free of small prime factors.

An obstacle in this step is a zero divisor in \( \Z/{\cal N} \)
appearing as the leading coefficient of some polynomials.
If additionally \( {\cal N} \) is not a perfect power,
we split \( {\cal N}  \) into factors which are co-prime
to each other.
This allows us to split the input
ideal. We then work on each factor ideal separately.

A degenerate case, when \( {\cal N} = N^e \) 
is a perfect power, must be handled separately.
For this case, we  run the Euclidean algorithm
on \( a(x) \) and \( b(x) \) over \( \Z/{N} \)
(rather than over \( \Z/{\cal N} \)), 
producing an ideal of the form
\begin{equation}\label{eq:onegennotN}
 \langle N^e, a_1 (x), N d_1(x), N d_2(x), \cdots \rangle,
\end{equation}
where only one  generator is not a multiple of
\( N \). We add \( f(x) \) into the ideal, and
it guarantees that \( a_1(x) \) is a factor of \( f(x) \)
over \( \Z/N \).

It may happen that \( a(x) \) and \( b(x) \) are both multiples of a
factor of \( {\cal N} \). We take such integer factors out of the
ideal, and collect them in an integer \( N_0 \).
In the end, we factor the input ideal into \( N_0 \) and
a product of factors in special forms:
\begin{equation}\label{eq:facinput}
 \langle {\cal N}, a(\alpha), b(\alpha) \rangle=
  N_0 \prod_{i} \langle N_i, a_i (\alpha) \rangle \prod_{j} \langle N_j^{e_j}, a_j (\alpha),
  N_j d_{ 1 j}(\alpha), N_j d_{ 2 j}(\alpha), \cdots \rangle.
\end{equation}

{\bf Stage 2}
The second stage of our algorithm will handle
ideals in the form of \eqref{eq:onegennotN}.
One can clearly read off a factor \( \langle N, a_1(\alpha) \rangle \)
from the ideal that is in two-element representation.
The other factors will also likely contain \( N \).
How do we multiply \( \langle N, a_1(\alpha) \rangle \langle N, a_2(\alpha) \rangle \)
into a two-element form? Preferably the product is
\( \langle N^2, a_1(\alpha) a_2(\alpha) \rangle \),
but this is not always true.
\begin{example}\label{ex:1}
  Consider the 8-th cyclotomic field with ring of
  integers \( \Z[x]/(x^4+1)  \); in this example \( x \) stands for
  its image \( \alpha \) in this ring.
  The product of ideals \( \langle 65, x^2 + 268 \rangle \) and \( \langle 65, x^2 - 268 \rangle \)
  is:
  \begin{align*}
    & \langle 65^2, 65(x^2+268), 65(x^2-268), (x^2+268)(x^2-268) \rangle \\
    =& \langle 65^2, 65 \times 2 \times 268, 65(x^2-268), -71825 \rangle\\
    =& \langle 65 \rangle,
  \end{align*}
  but
\[  \langle 65^2, (x^2+268)(x^2-268) \rangle = \langle 65^2, -71825 \rangle
  = \langle 65^2, -17 \times 65^2 \rangle = \langle 65^2 \rangle. \]
  If we add  some
  multiples of \( 65 \) to the second generator,
  \[ \langle 65, x^2 + 268 \rangle=\langle 65, x^2+8 \rangle, \langle 65, x^2 - 268 \rangle
    = \langle 65, x^2-8 \rangle,\]
  then
\[  \langle 65^2, (x^2+8)(x^2-8) \rangle = \langle 65^2, -65 \rangle = \langle 65 \rangle. \]
In the terminology of \cite[Section 6.3]{PZ97},
\( \langle 65, x^2+8 \rangle \) is a normal presentation of the ideal,
while \( \langle 65, x^2+268 \rangle \) is not.
In fact,  random multiples will work,
which explains the existence of randomized algorithms.
\end{example}
For each factor extracted in the second stage, we construct a normal
presentation with respect to the primes dividing \( N \). This
construction is deterministic even when \( N \) is composite and
its prime factorization is unknown.
 See Sections~\ref{sec:dedekind} and
\ref{sec:powerandmulti}.

{\bf Stage 3} After the second stage,
the algorithm splits the input ideal uniformly:
\[ \langle {\cal N}, a(\alpha), b(\alpha) \rangle
  = N_0 \prod_{i} \langle N_i^{k_i}, c_i (\alpha) \rangle, \]
where all the \( N_i \)'s are pairwise co-prime to each other.
Algorithm~\ref{alg:stage3} uses the Chinese remainder theorem
to combine these factors into a single two-element representation
(Lemma~\ref{lem:relativeprime}). It then multiplies both generators
by \( N_0 \).

We write \( \tilde{O}(T) \) for \( O(T (\log T)^{O(1)}) \), and count
bit operations.

\begin{theorem}[Main]\label{thm:main}
  Let \( f(x) \) be a monogenic polynomial of degree \( n \), let
  \( {\cal N} > 1 \) be an integer, and let \( a(x), b(x) \in \Z[x] \)
  have degree less than \( n \) and coefficients of absolute value less
  than \( {\cal N} \).  Let
  \( {\cal I} = \langle {\cal N}, a(\alpha), b(\alpha) \rangle \subseteq O_K \), and let
  \( m \) be the largest exponent in the factorization of \( {\cal I} \) into
  prime ideals (\( m = 1 \) if \( {\cal I} = O_K \)).
  There is a deterministic algorithm that finds an integer \( M \)
  dividing \( {\cal N} \) and a polynomial \( c(x) \in \Z[x] \) of
  degree less than \( n \) such that
  \[ {\cal I} = M O_K + c(\alpha) O_K \]
  in time
  \[ \tilde{O}( n^2\, \omega({\cal N})\, m^2 \log {\cal N} ), \]
  where \( \omega ({\cal N}) \) denotes the number of distinct prime
  factors of \( {\cal N} \).
  When \( {\cal N} \) has no repeated prime factor,
  the algorithm runs in time
  \( \tilde{O}(n^2\, \omega({\cal N}) \log {\cal N}) \).
\end{theorem}

In Theorem~\ref{thm:main}, the output is a pair of generators for
\( {\cal I} \). Normality is used in the second stage to multiply
intermediate factors with the same integer generator. The final
combination uses only the coprimality of the integer generators.
In fact, finding a factor of \( {\cal N} \) only slows
down our algorithm.
The algorithm does not decide whether \( {\cal N} \) has a repeated
prime factor, which is not known to be possible in polynomial time;
it only detects perfect powers \cite{Bernstein98}.
The second bound in Theorem~\ref{thm:main} holds whenever the first
stage leaves R2 empty, which is always the case when \( {\cal N} \)
has no repeated prime factor.
Since \( m \leq n \log_2 {\cal N} \), the running time is polynomial
in \( n \) and \( \log {\cal N} \).
Note that \( \omega({\cal N}) \)
is less than \(  \log {\cal N} \), and has average order
\( \log\log {\cal N} \)~\cite{HW08}.  If \( {\cal N} \) is the minimal
positive integer in \( {\cal I} \), then every prime factor of \( {\cal N} \)
lies below a prime ideal dividing \( {\cal I} \), so
\( \omega({\cal N}) \) is at most the number of distinct
prime ideals dividing \( {\cal I} \).
Let us compare with the randomized algorithm in the same model,
where the ideal is given by three generators.
The randomized algorithm picks a random element \( r \in {\cal I} \), and has
to verify that \( {\cal I} = {\cal N} O_K + r O_K \), namely, that
\( a(\alpha) \) and \( b(\alpha) \) are in the ideal
\( \langle {\cal N}, r \rangle \).  This is usually done by computing the Hermite
Normal Form, modulo \( {\cal N} \), of the \( n \times n \) integral
matrix of the multiplication by \( r \) \cite{Cohen93}, which takes
\( O(n^3) \) operations in \( \Z/{\cal N} \), hence time
\( \tilde{O}(n^3 \log {\cal N}) \) for every random choice.
The test in \cite[Theorem 6.5]{FHHJ17} decides, by one inversion in
\( K \) and a gcd, whether a pair is a normal presentation of the ideal
it generates; it needs extra work to
decide whether that ideal is \( {\cal I} \).
In contrast, if \( {\cal N} \) is not a perfect power and no zero
divisor is met, our algorithm consists of
one run of the Euclidean algorithm and takes time
\( \tilde{O}(n^2 \log {\cal N}) \), which saves a factor of \( n \).
Timings of a prototype implementation are reported in
Section~\ref{sec:exp}.
In general our algorithm remains faster as long as
\( \omega({\cal N})\, m^2 \) is smaller than \( n \).
The cost of the randomized algorithm does not depend on \( m \),
so it should be preferred for ideals that are divisible by a high
power of a prime ideal.
On the other hand, its expected number of trials depends on the prime
ideals of small norm dividing \( {\cal N} \), whereas our bound holds
for every input. Our algorithm
  works for all ideals in monogenic fields, covering
  cyclotomic fields---the most important case
  for ideal-based cryptography.
It also works for non-monogenic \( f(x) \) if \( {\cal N} \) is
relatively prime to the index \( [O_K:\Z[\alpha]] \).
Moreover, if the algorithm fails for this reason,
it will provide a certificate that \( f(x) \)
is not monogenic.  See Remark~\ref{rem:nonmono}.

\section{Mathematical Preparation}

It is well-known \cite{Cohen93} that
a prime ideal containing \( p \) has two-element
representation \( \langle p, t(\alpha) \rangle \) if
\( p \) does not divide the index of \( [O_K:\Z[x]/(f)] \).
 Here \( t(x) \) is an irreducible factor of \( f(x) \)
over \( \F_p \).
In this section, we present a few facts about ideals
in number fields. They are known in different forms in the literature.
We include some proofs here for completeness.

\subsection{The minimal positive integer in an ideal}

Given an ideal \( I \) in \( O_K \), we denote the
minimal positive integer in the ideal by \( \MPI(I) \).
It divides the norm and can be computed by
Hermite Normal Form of the lattice corresponding to
\( I \). If the number field is monogenic and
the ideal is principal \( \langle a(\alpha) \rangle \), it can
also be computed by the more efficient Euclidean algorithm.
First we find the inverse of \( a(\alpha) \) in the field,
by running the Euclidean algorithm on \( f(x) \) and \( a(x) \)
in \( \Q[x] \). The inverse must be in the form
\( b(\alpha) \) for some polynomial \( b(x) \in \Q[x] \).
Let the number \( i\) be the smallest positive integer
such that \( i b(x) \in \Z[x] \).
It can be shown that \( i =  \MPI (\langle a(\alpha) \rangle)  \).
If an ideal is generated by a few elements \( \alpha_1,
\alpha_2, \cdots\), we can set \({\cal N} \) to be
\( \gcd(\MPI(\langle \alpha_1 \rangle), \MPI(\langle \alpha_2 \rangle), \cdots) \).

\subsection{When integer generators are relatively prime}

We now show that factorization of an integer
generator of an ideal leads to factorization of the ideal.

\begin{lemma} \label{lem:factorideal}
  Let \( I \) be an ideal in \( \Z[x] \). Suppose that
  an integer \( {\cal N} \) factors as \( {\cal N} = N_1 N_2 \),
  with \( \gcd(N_1, N_2) = 1 \). Then the ideal factors as
  \[ {\cal N} \Z[x] + I = (N_1 \Z[x]+ I) (N_2 \Z[x] + I)  \] 
\end{lemma}

\begin{proof}
It is clear that \( \mathrm{RHS} \subseteq \mathrm{LHS} \),
  since 
  \[  (N_1 \Z[x]+ I) (N_2 \Z[x] + I) = N_1 N_2\Z[x]
  + N_1 I \Z[x] + N_2 I \Z[x] + I^2. \]
   To see that
  \( \mathrm{LHS} \subseteq \mathrm{RHS} \), choose integers \( u, v \) with
  \( u N_1 + v N_2 = 1 \).  The integer \( {\cal N} = N_1 N_2 \) is in
  the RHS, and for any \( h(x) \in I \),
  \( h(x) = u N_1 h(x) + v N_2 h(x) \in N_1 I + N_2 I \),
  which is contained in the RHS.
\end{proof}

\begin{lemma}
\label{lem:relativeprime}
  Assume that two positive integers \( N_1, N_2 \)
  are relatively prime. Let
  \[ u_1 =  N_1^{-1} \pmod{N_2 } \quad\text{and}\quad
  u_2 =  N_2^{-1} \pmod{N_1 }.\]
Then two ideals in \( \Z[x] \) can be multiplied
in polynomial time as
  \[ \langle N_1, h_1(x) \rangle \langle N_2, h_2(x) \rangle = \langle N_1 N_2, g(x) \rangle, \]
  where \( h_1(x), h_2(x) \) are in \( \Z[x] \), and
  \( g(x) = N_2 u_2 h_1(x) + N_1 u_1 h_2(x) \).
\end{lemma}

\begin{proof}
Essentially we find a polynomial \( g(x) \) so that
\[ g(x) = h_1(x) \pmod{N_1},\ \ \  g(x) = h_2(x) \pmod{N_2}, \]
using the Chinese Remainder Theorem,
and then apply Lemma~\ref{lem:factorideal}.
\end{proof}

\section{A Euclidean algorithm for integral polynomials}

The Euclidean algorithm for polynomials over a field also runs over
\( \Z/N \), as long as every leading coefficient that has to be inverted
is a unit; a leading coefficient that is not a unit reveals a divisor of
\( N \).  Computing in this way modulo a composite number of unknown
factorization is classical; see \cite{FHS20} for a detailed treatment
over \( \Z/N \).  We recall the algorithm in the form we need, as a
statement about ideals of \( \Z[x] \).

For \( g(x), h(x) \in \Z[x] \), we say that \( g(x) \) is
\emph{monic modulo \( N \)} if its image in \( (\Z/N)[x] \) is monic,
and that \( g(x) \) \emph{divides \( h(x) \) over \( \Z/N \)} if
\( h(x) \in \langle N, g(x) \rangle \).

\begin{lemma}[Euclidean algorithm modulo \( N \)]\label{lem:polygcd}
  Let \( N > 1 \) be an integer that is not a perfect power, let
  \( e \geq 1 \), and let \( a(x), b(x) \in \Z[x] \) have degree at most
  \( n \), with \( a(x) \not\equiv 0 \pmod{N} \).  There is a
  deterministic algorithm, using \( O(n^2) \) operations in
  \( \Z/N^e \), that
  returns one of the following.
  \begin{enumerate}
  \item[(i)] A factorization \( N = N_1 N_2 \) with \( N_1, N_2 > 1 \)
    and \( \gcd(N_1, N_2) = 1 \).
  \item[(ii)] Polynomials \( c(x), d(x), u(x), v(x) \in \Z[x] \) of
    degree at most \( 2n \) such that \( c(x) \) is monic modulo \( N \),
    \begin{equation}\label{eq:polygcd}
      \langle N^e, a(x), b(x) \rangle = \langle N^e, c(x), N d(x) \rangle
      \quad\text{in } \Z[x],
    \end{equation}
    and
    \( u(x) a(x) + v(x) b(x) \equiv c(x) \pmod{N^e} \).
  \end{enumerate}
  In case (ii) the following hold.
  \begin{enumerate}
  \item[(a)] \( \langle N, a(x), b(x) \rangle = \langle N, c(x) \rangle \);
    in particular \( c(x) \) divides \( a(x) \) and \( b(x) \) over
    \( \Z/N \).
  \item[(b)] For every prime \( p \mid N \), the reduction of \( c(x) \)
    modulo \( p \) is the monic greatest common divisor of \( a(x) \)
    and \( b(x) \) in \( \F_p[x] \).
  \item[(c)] \( \langle N, a(x), b(x) \rangle = \langle 1 \rangle \) if
    and only if \( c(x) \equiv 1 \pmod{N} \).
  \end{enumerate}
\end{lemma}

\begin{proof}
  All polynomials are kept with coefficients reduced modulo \( N^e \).
  To \emph{normalize} a polynomial \( r(x) \not\equiv 0 \pmod N \), let
  \( \ell \) be the leading coefficient of \( r(x) \) modulo \( N \) and
  compute \( g = \gcd(\ell, N) \).  If \( g > 1 \), then \( 1 < g < N \)
  and we go to the last paragraph of the proof.  Otherwise we multiply
  \( r(x) \) by the inverse of \( \ell \) modulo \( N^e \), which makes
  it monic modulo \( N \).

  Exchanging \( a(x) \) and \( b(x) \) if necessary, let \( r_0 = a \)
  and \( r_1 = b \), where the degree of \( r_0 \) modulo \( N \) is at
  least that of \( r_1 \), and normalize \( r_0 \).  For
  \( i = 1, 2, \ldots \), as long as \( r_i \not\equiv 0 \pmod N \), we
  normalize \( r_i \), let \( q_i \in \Z[x] \) be a lift of the quotient
  of \( r_{i-1} \) by \( r_i \) in \( (\Z/N)[x] \), and put
  \( r_{i+1} = r_{i-1} - q_i r_i \).  The degree of \( r_i \) modulo
  \( N \), which we denote by \( k_i \), strictly decreases, so there is
  a \( \lambda \leq n+1 \) with \( r_{\lambda+1} \equiv 0 \pmod N \).  We return
  \( c = r_\lambda \) and \( d = r_{\lambda+1}/N \).  Each step replaces the pair
  \( (r_{i-1}, r_i) \) by \( (r_i, r_{i-1} - q_i r_i) \) or multiplies
  one of its members by a unit of \( \Z/N^e \); both operations are
  invertible, so the pair generates the same ideal modulo \( N^e \)
  throughout.  This gives \eqref{eq:polygcd}, and performing the same
  operations on the pairs \( (1,0) \) and \( (0,1) \) gives \( u \) and
  \( v \).

  The polynomial \( r_i \) may have degree larger than \( k_i \),
  because its coefficients above degree \( k_i \) are multiples of
  \( N \) but need not vanish modulo \( N^e \).  Since
  \( \deg r_{i+1} \leq \max(\deg r_{i-1},\, k_{i-1} - k_i + \deg r_i) \)
  and the \( k_i \) decrease, induction gives
  \( \deg r_i + k_{i-1} \leq 2n \) for all \( i \geq 1 \).  Hence all
  degrees are at most \( 2n \), step \( i \) costs
  \( O((\deg q_i + 1)\, n) \) operations in \( \Z/N^e \), and
  \( \sum_i \deg q_i \leq n \).

  For (a), note that \( N d(x) \in \langle N \rangle \) and add \( N \)
  to both sides of \eqref{eq:polygcd}.  Reducing (a) modulo a prime
  \( p \mid N \) shows that \( a(x) \) and \( b(x) \) generate the same
  ideal of \( \F_p[x] \) as \( c(x) \), which is monic modulo \( p \);
  this is (b).  If \( c(x) \equiv 1 \pmod N \), then
  \( 1 \in \langle N, c(x) \rangle \).  Otherwise \( c(x) \) is monic of
  positive degree modulo any prime \( p \mid N \), so
  \( \langle N, c(x) \rangle \subseteq \langle p, c(x) \rangle \neq
  \langle 1 \rangle \).  This proves (c).

  It remains to treat a divisor \( g \) of \( N \) with \( 1 < g < N \).
  Factor refinement \cite{BDS93,Bernstein05} applied to
  \( N = g \cdot (N/g) \) produces, in time polynomial in \( \log N \),
  pairwise coprime integers \( \pi_1, \ldots, \pi_\nu > 1 \) and positive
  exponents with \( N = \pi_1^{j_1} \cdots \pi_\nu^{j_\nu} \), such that
  \( g \) and \( N/g \) are products of powers of the \( \pi_i \).  If
  \( \nu = 1 \), then \( g \) and \( N/g \) are both powers of \( \pi_1 \)
  with positive exponent, so \( N = \pi_1^{j_1} \) with \( j_1 \geq 2 \),
  contradicting the assumption that \( N \) is not a perfect power.
  Hence \( \nu \geq 2 \), and we return \( N_1 = \pi_1^{j_1} \) and
  \( N_2 = N/N_1 \).
\end{proof}

In case (ii) we call \( c(x) \), taken modulo \( N \), the
\emph{greatest common divisor of \( a(x) \) and \( b(x) \) over
  \( \Z/N \)}.  For \( e = 1 \), Equation~\eqref{eq:polygcd} reads
\( \langle N, a(x), b(x) \rangle = \langle N, c(x) \rangle \), which
solves the two-element representation problem for this ideal.  If both
\( a(x) \) and \( b(x) \) are multiples of \( N \), then the ideal is
\( \langle N \rangle \) for \( e = 1 \), and for \( e \geq 2 \),
\( \langle N^e, a(x), b(x) \rangle
  = N \langle N^{e-1}, a(x)/N, b(x)/N \rangle \).
Repeating this, we take the largest power of \( N \) dividing both
\( a(x) \) and \( b(x) \) out of the ideal as an integer factor,
and the lemma can be applied to the other factor.

What if a non-trivial zero divisor in \( \Z/N \) is found?
The following theorem shows that  we can
always factor and normalize the input ideal:

\begin{theorem}\label{thm:factor}
  There is a polynomial time algorithm that,
  given two integral polynomials \( a(x ) \) and
  \( b(x ) \), a positive integer \( N  \)
  that is not a perfect power, and a positive integer \( e \),
  finds a positive integer \( N_0 \) and
  factors the ideal \( \langle N^e, a(x), b(x) \rangle \subseteq \Z[x] \)
  into a product as
  \[ \langle N^e, a(x), b(x) \rangle
    = N_0 \prod_{i} \langle N_i^{e_i}, a_i(x), N_i d_i (x) \rangle \]
  such that none of \( N_i \)'s is a perfect power,
  every \( a_i(x) \) is monic modulo \( N_i \),
  for any \( i \not= j \), \( \gcd(N_i, N_j) = 1\),
  and \( N_0 \prod_i N_i^{e_i} = N^e \).
\end{theorem}

\begin{proof}
  We keep an integer \( N_0 \), a list of integers that remain to be
  treated, and a list of ideals.  At the beginning \( N_0 = 1 \), the
  first list contains \( N^e \) only, and the second list is empty.
  Take an integer from the first list and write it as \( P^E \), where
  \( P \) is not a perfect power.  Let \( k \leq E \) be the largest
  integer such that \( P^k \) divides both \( a(x) \) and \( b(x) \).

  If \( k = E \), then
  \( \langle P^E, a(x), b(x) \rangle = \langle P^E \rangle \), and we
  multiply \( N_0 \) by \( P^E \).
  Otherwise
  \[ \langle P^E, a(x), b(x) \rangle
    = P^k \langle P^{E-k}, a(x)/P^k, b(x)/P^k \rangle, \]
  and one of \( a(x)/P^k \) and \( b(x)/P^k \) is not a multiple of
  \( P \).  We run the algorithm from Lemma~\ref{lem:polygcd} on
  \( a(x)/P^k \), \( b(x)/P^k \), \( P \) and \( E-k \).
  In case (ii) we find \( c(x) \), monic modulo \( P \), and \( d(x) \)
  such that
  \( \langle P^{E-k}, a(x)/P^k, b(x)/P^k \rangle
  = \langle P^{E-k}, c(x), P d(x) \rangle \);
  we multiply \( N_0 \) by \( P^k \) and put the ideal
  \( \langle P^{E-k}, c(x), P d(x) \rangle \) into the second list.
  In case (i) we find a factorization \( P = P_1 P_2 \),
  where \( P_1, P_2 > 1 \) are co-prime.  Then
  \[ \langle P^E, a(x), b(x) \rangle = \langle P_1^{E}, a(x), b(x) \rangle
    \langle P_2^{E}, a(x), b(x) \rangle \]
  by Lemma~\ref{lem:factorideal}; we put \( P_1^E \) and \( P_2^E \)
  into the first list and leave \( N_0 \) unchanged.

  Throughout, the integers in the first list and the integer
  generators of the ideals in the second list are pairwise co-prime,
  and their product times \( N_0 \) is \( N^e \).  Every occurrence of
  case (i) increases the number of these integers, which is at most
  the number of distinct prime factors of \( N \).  Hence the
  first list becomes empty after polynomially many steps, and the
  second list is the desired factorization.
\end{proof}

We can add \( f(x) \) into the ideal,
run the
algorithm from Lemma~\ref{lem:polygcd} on \( a_i(x) \),
\( f(x) \) and \( N \),
and derive a corollary:

\begin{corollary}
\label{cor:gcd}
  There is a polynomial time algorithm that,
  given two integral elements \( a(\alpha ) \) and
  \( b(\alpha) \) in \( O_K \), a positive integer \( N  \)
  that is not a perfect power, and a positive integer \( e \),
  factors
  the ideal \( \langle N^e, a(\alpha), b(\alpha) \rangle \subseteq O_K \)
  into a product as
  \begin{equation}
\label{eq:5}
 \langle N^e, a(\alpha), b(\alpha) \rangle =
 N_0 \prod_{i} \langle N_i^{e_i}, a_{ i }(\alpha),
    N_i I_i \rangle  
  \end{equation}
  such that \( N_0 \) is a positive integer and
  \( N_0 \prod_i N_i^{e_i} = N^e \);
  none of \( N_i \)'s is a perfect power;
  \( I_i = \langle d_{i1} (\alpha), d_{i2} (\alpha) \rangle \)
  is an ideal in \( O_K \); and
  for any \( i \not= j \), \( \gcd(N_i, N_j) = 1\);
  and \( a_{i}(x) \) is monic modulo \( N_i \) of degree less than
  \( n \), and divides \( f(x) \) over \( \Z/N_i \).
\end{corollary}

\begin{algorithm}[h]
  \caption{Sketch of the algorithm in the first stage}
  \label{alg:stage1}
    \KwIn{ \( f(x) \) monogenic of degree \( n \),
    an integer \( {\cal N}\in \Z \)
    and two elements \( a(x), b(x) \) in \( \Z[x] \)
    of degree less than \( n \).}
  \KwOut{ An integer \( N_0 \) and two stacks: R1 and R2.  }
  \tcc{Elements in R1 are ideals
    of the form \( \langle N_i, a_i(\alpha) \rangle \), elements in R2 have the form
    \( \langle N_i^{e_i}, a_i(\alpha),
    N_i d_{i1}(\alpha), N_i d_{i2}(\alpha) \rangle \)
    with \( e_i \geq 2 \) and \( a_i(x) \) monic modulo \( N_i \)
    of degree less than \( n \),
    and all \( N_i \) are pairwise co-prime.
    The product of \( N_0 \) and all ideals in R1 and R2 is equal to
  \( \langle {\cal N}, a(\alpha), b(\alpha) \rangle \)}
  \( N_0 \gets 1 \)\;
  Initialize empty stacks S, R1 and R2\;
  Push \( {\cal N} \) into S\;
  \While{S is not empty}{
      Pop an integer  from S, and write it as \(N^e\) so
      that \( N \) is not a perfect power\;
      \( k \gets \) the largest integer \( k \leq e \) such that
      \( N^k \) divides both \( a(x) \) and \( b(x) \)\;
      \If{\( k = e \)}{
        \tcc{the ideal \( \langle N^e, a(\alpha), b(\alpha) \rangle \) is \( \langle N^e \rangle \)}
        \( N_0 \gets N_0 \times N^e \), and go to the next iteration\;
      }
    \Try{}{
      Run the algorithm from Lemma~\ref{lem:polygcd}
      on \( a(x)/N^k, b(x)/N^k\), \( N \) and \( e-k \) to find
      \( c(x), d(x) \) so that
      \(\langle N^{e-k}, a(x)/N^k, b(x)/N^k \rangle
      = \langle N^{e-k}, c(x), N d(x) \rangle\)\;
      \uIf{\( e - k = 1 \)}
        {Push \( \langle N, c(\alpha) \rangle \) into R1\;
        }
      \Else{Run the algorithm  from Lemma~\ref{lem:polygcd}
          on \(c(x)\), \(f(x)\), \( N \) and \( e-k \) to find
      \( a_1(x), d_1(x) \) so that
      \(\langle N^{e-k}, c(x), f(x) \rangle
      = \langle N^{e-k}, a_1 (x), N d_1 (x) \rangle\)\;
          Push \( \langle N^{e-k}, a_1 (\alpha), N d(\alpha), N d_1 (\alpha) \rangle \)
          into R2\;
        }
      \( N_0 \gets N_0 \times N^k \)\;
    }
    \Catch{\( N \) is factored as \( N = N_1 N_2 \) with
      \( \gcd(N_1, N_2) = 1 \)}
    { \tcc{\( N_0 \) is not changed; \( k \) is found again for each
        factor}
      Push \( N_1^{e}\) and \( N_2^{e} \) into the stack S\;
    }
    }
\end{algorithm}

The first stage of the algorithm will factor the input
ideal as in Equation~\eqref{eq:facinput}, and is described
in pseudo-code in Algorithm~\ref{alg:stage1}. 
Note that we may write \( a_{i} (x) \), in many ways,
as a monic polynomial dividing \( f(x) \) over \( \Z/N_i \), plus
a multiple of \( N_i \) in \( \Z[x] \).
This flexibility plays an essential
role in later stages of the algorithm.
For details see Theorem~\ref{thm:tech}. 
For a factor in the RHS of Equation~\eqref{eq:5} with \( e_i =1 \),
it is already in the two-element form.
While the other factors in the RHS with \( e_i > 1 \)
are not in two-element representation,
we can immediately read off a factor \( \langle N_i, a_{i}(\alpha) \rangle \)
with two generators.
If \( a_{i}(\alpha) = 1\pmod{N_i} \), then
\(  \langle N_i^{e_i}, a_{ i }(\alpha), N_i I_i \rangle\) is
    the entire ring \( O_K \), according to the following lemma.

\begin{lemma}\label{lem:unit}
  If \( A \in ( \Z/N )^* \), then for any \( g(x) \),
  \( A+N g(x) \) is a unit in \( \Z/N^s [x] \) for
  any positive integer \( s \).
\end{lemma}

\begin{proof}
  Let \( A B = 1+N C \) and \( u(x) = N(B g(x)+C) \). Then
  \[ (A+N g(x))\cdot B\sum_{j=0}^{s-1}(-u(x))^j
    = (1+u(x))\sum_{j=0}^{s-1}(-u(x))^j
   = 1-(-u(x))^s \equiv 1 \pmod{N^s}. \]
\end{proof}

If \( a_{i}(\alpha) \not= 1 \pmod{N_i} \),
we need to compute
\( \langle N_i^{e_i}, a_{ i }(\alpha), N_i I_i \rangle/\langle N_i, a_{i}(\alpha) \rangle \).
In the next two sections, we show how to
find a proper member in  \( a_{i} (x) + N_i \Z[x] \) 
so that ideal division and multiplication
behave nicely.

\section{Generalized Dedekind's criterion}
\label{sec:dedekind}

Let \( f(x) \in \Z [x] \) be a monic irreducible polynomial,
and \( N > 1 \) be an integer that is not a perfect power.
Let \( a_1(x ) \in \Z[x] \) be a monic proper factor
of \( f(x) \) over \( \Z/N \) of degree at least one.
We can write
\begin{equation}\label{eq:1} 
 f(x) = f_1 (x) N + a_1 (x) A(x),  
\end{equation}
where \( A(x) \) is also monic of degree at least one.

\begin{lemma}[Dedekind's criterion, necessity]\label{lem:gended}
Let \( N>1 \) be an integer and let \( c(x)\in\Z[x] \) be monic of
degree at least one, dividing \( f(x) \) over \( \Z/N \).
If \( f(x) \in \langle N, c(x)\rangle^2 \), then
\( N \mid [O_K:\Z[\alpha]] \), and \( f(x) \) is not monogenic.
\end{lemma}

This is basically part of Dedekind's criterion
\cite{Dedekind1878,Uchida77}, \cite[page 31]{GW21}.
In its modern form \cite[Theorem 6.1.4]{Cohen93}, the criterion
reads as follows.  Let \( p \) be a prime, and let
\( f(x) \equiv \prod_i \varphi_i(x)^{e_i} \pmod{p} \), where the
\( \varphi_i(x) \in \Z[x] \) are monic, and irreducible and distinct
modulo \( p \).  Put \( g(x) = \prod_i \varphi_i(x) \), let
\( h(x) \) be a monic lift of \( f(x)/g(x) \) modulo \( p \), and
let \( T(x) = (g(x) h(x) - f(x))/p \).  Then \( p \) does not divide
\( [O_K:\Z[\alpha]] \) if and only if \( T(x) \), \( g(x) \) and
\( h(x) \) have no common factor modulo \( p \).
Equivalently \cite{Uchida77}, \( \Z[\alpha] = O_K \) if and only if
\( f(x) \) is not in the square of any maximal ideal of \( \Z[x] \).
For \( N = p \) and \( c(x) = \varphi_i(x) \), the condition
\( f(x) \in \langle p, \varphi_i(x) \rangle^2 \) of
Lemma~\ref{lem:gended} says exactly that \( \varphi_i(x) \) is a
common factor of \( T(x) \), \( g(x) \) and \( h(x) \) modulo \( p \).
Most of the proofs in the literature
assume that \( N \) is a prime. For the purpose of
self-containment, we present a constructive proof
for any integer \( N > 1 \).

\begin{proof}
We can write
\[ f(x)= c^2 (x) B(x) + N^2 g_1(x) + N c(x) g_2(x) \]
with \( g_1(x) \) and \( g_2(x) \) in \( \Z[x] \),
\( \deg (g_1) < 2 \deg c \), \( \deg (g_2) <  \deg c \),
(reduce \( g_1 \) modulo \( c^2 \),
and \( g_2 \) modulo \( c \) if necessary),
and \( B(x) \) monic.
Let \( \alpha  \) be one of the roots of \( f(x) \).
Multiplying the RHS by \( B(x) \), we obtain
\[ c^2 (\alpha) B^2(\alpha) + N^2 g_1(\alpha)  B(\alpha) + N c(\alpha) B(\alpha) g_2(\alpha) = 0. \]
Dividing by \( N^2 \),
\[ ( ( c(\alpha) B(\alpha) )/N )^2  +
g_1(\alpha) B(\alpha) +
g_2(\alpha) ( c(\alpha) B(\alpha) )/N = 0. \]
We conclude that
\( ( c(\alpha) B(\alpha) )/N \)
is integral in \( \Q[x]/(f(x)) \).
It is also clear that \( j ( c(\alpha) B(\alpha) )/N \)
is not in \( \Z[\alpha] \) for any \( 1\leq j < N \),
since \( c(x) B(x) \)
is monic with degree less than the degree of \( f(x) \).
Therefore \( N \mid [O_K:\Z[\alpha]] \), and \( f(x) \) is not monogenic;
the element \( ( c(\alpha) B(\alpha) )/N \) is a certificate.
\end{proof}

\begin{lemma}\label{lem:split}
Assume that \( N \) is co-prime to the index \( [O_K:\Z[\alpha]] \)
(for instance, that \( f(x) \) is monogenic),
and write \( f(x) = a_1(x) A(x) + N f_1(x) \) as above.
Using at most \( \deg a_1(x) + 2 \) calls to the algorithm of
Lemma~\ref{lem:polygcd}, we either find a non-trivial factor of \( N \),
or monic \( a_2(x), a_3(x)\in\Z[x] \) with \( a_1(x) \equiv a_2(x)a_3(x) \pmod N \) and
\begin{enumerate}
\item \( \langle a_3(x), f_1(x), N\rangle = \langle 1\rangle \);
\item for every prime \( p\mid N \), every irreducible factor of
  \( a_2(x) \bmod p \) divides \( f_1(x) \bmod p \);
\item \( \langle a_2(x), A(x), N\rangle = \langle 1\rangle \).
\end{enumerate}
\end{lemma}
\begin{proof}
Put \( h_0(x)=a_1(x) \). Given \( h_j(x) \), compute \( g_j(x)=\gcd(h_j(x),f_1(x)) \) over
\( \Z/N \) (normalised monic) and \( h_{j+1}(x)=h_j(x)/g_j(x) \) over \( \Z/N \);
stop when \( g_j(x)=1 \). Degrees strictly decrease, so there are at most
\( \deg a_1(x) \) steps. Let \( a_3(x) \) be the last \( h_j(x) \) and
\( a_2(x)=\prod_j g_j(x) \). Note that if \( f_1(x) = 0 \pmod{N} \),
then \( a_2(x) = a_1(x) \) and
\( a_3(x) = 1 \). More generally,
\[ a_1 (x) = g_0 (x) h_1(x) = g_0(x) g_1(x) h_2(x) =
   ( \prod_i g_i(x) ) a_3(x) = a_2(x) a_3(x) \pmod{N}\]
Then (1) is item (c) of
Lemma~\ref{lem:polygcd}, and (2) holds because each \( g_j(x) \) divides
\( f_1(x) \) over \( \Z/N \).
For (3), compute \( c(x)=\gcd(a_2(x),A(x)) \) over \( \Z/N \). If \( \deg c(x)>0 \),
compute \( c'(x)=\gcd(c(x),f_1(x)) \). When no zero divisor is met, the
computation reduces modulo each \( p\mid N \) to the Euclidean algorithm
over \( \F_p \), so by (2) \( c'(x) \) has positive degree. Now \( c'(x) \)
divides \( a_1(x) \), \( A(x) \) and \( f_1(x) \) over \( \Z/N \), hence
\( f(x) = a_1(x)A(x)+Nf_1(x)\in\langle N,c'(x)\rangle^2 \), contradicting
Lemma~\ref{lem:gended}, and presenting a certificate that \( f(x) \)
is not monogenic.
\end{proof}

Although we may find some factors of
\( a_1 (x) \) over \( \Z/N \) during the algorithm,  
there is in general no efficient way to factor 
polynomials when \( N \) is not a prime.
In fact, even taking square roots modulo \( N \) is
as hard as factoring \( N \).

\begin{theorem}\label{thm:tech}
In the situation of Lemma~\ref{lem:split}, put \( \tilde a(x) = a_1(x) \) if
\( a_2(x) = 1 \) and \( \tilde a(x) = a_1(x) + N a_3(x) \) otherwise, and
\( \tilde f_1(x) = (f(x)-\tilde a(x) A(x))/N \). Then \( \tilde a(x) \) is monic,
\( \tilde a(x)\equiv a_1(x) \pmod N \),
\begin{equation}\label{eq:3}
 f(x) = \tilde f_1(x) N + \tilde a(x) A(x) ,
\end{equation}
and \( \langle \tilde a(x), \tilde f_1(x), N\rangle = \langle 1\rangle \) in \( \Z[x] \).
\end{theorem}
\begin{proof}
If \( a_2(x)=1 \) this is (1). Otherwise \( \deg a_3(x)<\deg a_1(x) \), so
\( \tilde a(x) \) is monic, and \( \tilde f_1(x) = f_1(x) - a_3(x)A(x) \).
If the ideal were proper it would lie in a maximal ideal
\( \langle p,\varphi(x)\rangle \) with \( p\mid N \) and \( \varphi(x) \)
irreducible mod \( p \). Then \( \varphi(x) \mid a_2(x)a_3(x) \) and
\( \varphi(x)\mid f_1(x)-a_3(x)A(x) \) mod \( p \). If \( \varphi(x)\mid a_3(x) \), then
\( \varphi(x)\mid f_1(x) \), contradicting (1). Otherwise \( \varphi(x)\mid a_2(x) \),
so \( \varphi\mid f_1(x) \) by (2), hence \( \varphi\mid a_3(x)A(x) \) and so
\( \varphi(x)\mid A(x) \), contradicting (3).
\end{proof}

\section{Multiplying ideals in two-element representations}
\label{sec:powerandmulti}
We now return to an ideal
factor in the RHS of Equation~\eqref{eq:5}.
We omit the subscript for convenience,
and denote it by \( J =  \langle N^e, a(\alpha), N I \rangle \). 
Recall that \( N \) is not a perfect power.
We find \( \tilde{a}(x) \) satisfying
the condition in Equation~\eqref{eq:3},
and read off a factor \( \langle N, \tilde{a} (\alpha) \rangle \) from \( J \).
We also find \( g(\alpha) \)
so that \( a(\alpha) = \tilde{a}(\alpha) + N g(\alpha) \).
According to Theorem~\ref{thm:tech}, we may write
\begin{equation}\label{eq:4} 
 f(x) = f_1 (x) N + \tilde{a} (x) A(x),
\end{equation}
with \( \langle f_1(x), \tilde{a}(x), N \rangle = \langle 1 \rangle \) in \( \Z[x] \)
(which implies that \( f_1(x) \not= 0 \pmod{N} \) ). 
We show that the condition will enable fast ideal arithmetic.

\begin{example}
  In Example~\ref{ex:1}, we can see that it is
  easier to multiply when using \( a(x) = x^2+8 \),
  rather than \(x^2 + 268\), as the second generator,
  even though \( \langle 65, x^2 + 8 \rangle = \langle 65, x^2 + 268 \rangle \).
  This is because
  \[ f(x)= x^4+1 = (x^2+8) (x^2-8)+65  \]
  so \( x^2+8 \) satisfies the conditions in
  Equation~\eqref{eq:4}, whereas \( x^2+268 \) does not,
  as
  \[ f(x) = x^4+1 = (x^2+268) (x^2-268)+65^2\times 17.  \]
\end{example}

\begin{lemma}\label{lem:genN}
  For any positive integer \( s \), \( N \in \langle N^s, \tilde{a}(\alpha) \rangle\).
\end{lemma}

\begin{proof}
Since \( \langle \tilde{a}(x), f_1 (x), N \rangle = \langle 1 \rangle \),
there are integer polynomials \( a'(x), f_1'(x) \)
and \( t(x) \) such that
\[ \tilde{a}(x) a'(x) + f_1(x) f_1'(x) + N t(x) = 1.  \]
We also have
\[ f_1 (\alpha) N + \tilde{a}(\alpha) A(\alpha) =0 \]
so
\[ f_1'(\alpha) f_1 (\alpha) N + f_1'(\alpha) \tilde{a}(\alpha) A(\alpha) =0 \]
\[ (1 - \tilde{a}(\alpha) a'(\alpha) - N t(\alpha)) N + f_1'(\alpha)
  \tilde{a}(\alpha) A(\alpha) =0 \]
\[ (1 -  N t(\alpha)) N = \tilde{a}(\alpha) a'(\alpha) N
  - f_1'(\alpha) \tilde{a} (\alpha) A(\alpha)  \]
\[ (1 -  N t(\alpha)) N = \tilde{a}(\alpha) ( a'(\alpha) N
  - f_1'(\alpha)   A(\alpha)  ) \]
The statement follows from the fact that
\[ 1- N t(\alpha)  \] is a unit in
\( \Z[\alpha]/(N^{s}) \), by Lemma~\ref{lem:unit}.
\end{proof}

The polynomials \( a'(x) \), \( f_1'(x) \) and \( t(x) \) can be
computed: the algorithm in Lemma~\ref{lem:polygcd}, run on
\( \tilde{a}(x) \), \( f_1(x) \) and \( N \) with \( e = 1 \), either
finds a non-trivial factor of \( N \), or returns \( u(x), v(x) \)
and \( c(x) \) with
\( u(x) \tilde{a}(x) + v(x) f_1(x) \equiv c(x) \pmod{N} \); here
\( c(x) \equiv 1 \pmod{N} \) by item (c) of the lemma, and we take
\( a'(x) = u(x) \) and \( f_1'(x) = v(x) \).
By setting \[ \bar{a}(\alpha) = (a'(\alpha) N
- f_1'(\alpha)  A(\alpha))(1 + N t(\alpha) + ( N t(\alpha) )^2
+ \cdots + ( N t(\alpha) )^{s-1}) \in O_K, \]
which is congruent to \( (a'(\alpha) N
- f_1'(\alpha)  A(\alpha))/(1 - N t(\alpha)) \pmod{N^s}, \)
we conclude:

\begin{corollary}\label{cor:div}
  For \( \tilde{a} (x) \) in Equation~\eqref{eq:4}
  and any \( s\geq 2 \),
  we can, in deterministic polynomial time, either find a non-trivial
  factor of \( N \), or find \( \bar{a}(\alpha) \) such that
  \[ \tilde{a}(\alpha)  \bar{a}(\alpha)  = N \pmod{N^s}.  \]
  Moreover
  \[ \langle N, \tilde{a}(\alpha) \rangle \langle N, \bar{a}(\alpha) \rangle = \langle N \rangle.   \]
\end{corollary}

\begin{proof}
  The congruence follows from the last equation in the proof of
  Lemma~\ref{lem:genN}.  For the second claim, write
  \( \tilde{a}(\alpha) \bar{a}(\alpha) = N + N^2 w(\alpha) \) with
  \( w(\alpha) \in O_K \), which is possible since \( s \geq 2 \).
  The product
  \[ \langle N, \tilde{a}(\alpha) \rangle \langle N, \bar{a}(\alpha) \rangle = \langle N^2,
    N \tilde{a}(\alpha), N \bar{a}(\alpha),
    \tilde{a}(\alpha) \bar{a}(\alpha) \rangle \]
  is contained in \( \langle N \rangle \), since each of the four generators is a
  multiple of \( N \).  It contains \( N^2 \) and
  \( N + N^2 w(\alpha) \), hence \( N \).
\end{proof}

Now we are ready to calculate \( J/ \langle N, \tilde{a}(\alpha) \rangle  \)
\begin{lemma}
\label{lem:other}
Let \( \tilde{a} (x) \) satisfy the conditions
in Equation~\eqref{eq:4}, and let \( e \geq 2 \).
Given an ideal \( J = \langle N^e, \tilde{a}(\alpha) + N g(\alpha), N I \rangle \),
where \( I = \langle d_1 (\alpha), d_2(\alpha), \cdots, d_k (\alpha) \rangle \),
there is a polynomial time algorithm that either finds a non-trivial
factor of \( N \), or calculates \( J/ \langle N, \tilde{a}(\alpha) \rangle \),
and normalizes it into a form
\( \langle N^e, a' (\alpha) , N I' \rangle, \)
where \( a' (x) \) is monic modulo \( N \) and divides \( f(x) \)
over \( \Z/N \), and \( I' \) has at most \( k+2 \) generators.
Moreover, \( J/ \langle N, \tilde{a}(\alpha) \rangle = O_K \) if and only if
\( a'(x) \equiv 1 \pmod{N} \).
\end{lemma}

\begin{proof}
  Calculate \( \bar{a} \) from \( \tilde{a} \) using the
  algorithm in Corollary~\ref{cor:div} with \( s=e+1 \).
  Then \( \tilde{a}(\alpha) \bar{a}(\alpha) = N (1 + N^e w(\alpha)) \)
  for some \( w(\alpha) \in O_K \); in particular the division of
  \( \tilde{a}(\alpha) \bar{a}(\alpha) \) by \( N \) is exact.
  By Corollary~\ref{cor:div}, the inverse of the ideal
  \( \langle N, \tilde{a}(\alpha) \rangle \) is \( \langle N, \bar{a}(\alpha) \rangle/N \), so
  we have
\begin{align*}
  &J/ \langle N, \tilde{a}(\alpha) \rangle\\
  =&  \langle N, \bar{a}(\alpha) \rangle\langle N^e, \tilde{a}(\alpha) +N g(\alpha), N I \rangle/N\\
  =&  \langle N^{e+1}, N \tilde{a}(\alpha) +N^2 g(\alpha),
     N^2 I, N^{e} \bar{a}(\alpha),
      \tilde{a}(\alpha) \bar{a}(\alpha)+ N g(\alpha)\bar{a}(\alpha),
     \bar{a}(\alpha) N I \rangle/N\\
  =&  \langle N^{e},  \tilde{a}(\alpha) +N g(\alpha),
     N I, N^{e-1} \bar{a}(\alpha),
     1 + N^e w(\alpha)
     + g(\alpha)\bar{a}(\alpha),
     \bar{a}(\alpha)  I \rangle\\
  =&  \langle N^{e},  \tilde{a}(\alpha) +N g(\alpha),
     1 + g(\alpha)\bar{a}(\alpha),
     \bar{a}(\alpha)  I, N^{e-1} \bar{a}(\alpha) \rangle.
\end{align*}
In the last step, we remove \( N^e w(\alpha) \) because \( N^e \) is
in the ideal, and we remove \( N I \) for the following reason: for
\( d(\alpha) \in I \), the ideal contains
\( \bar{a}(\alpha) d(\alpha) \), hence
\( \tilde{a}(\alpha) \bar{a}(\alpha) d(\alpha)
= N d(\alpha) + N^{e+1} w(\alpha) d(\alpha) \),
hence \( N d(\alpha) \).

To normalize the ideal, we keep all polynomials reduced modulo
\( f(x) \), and start with \( c(x) = \tilde{a}(x) + N g(x) \), which
is monic modulo \( N \) of degree less than \( n \).  For each of the
\( k+1 \) generators \( h(\alpha) \) among
\( 1 + g(\alpha)\bar{a}(\alpha) \) and
\( \bar{a}(\alpha) d_i(\alpha) \), \( 1 \leq i \leq k \), in turn, we
run the algorithm from Lemma~\ref{lem:polygcd} on \( c(x) \),
\( h(x) \), \( N \) and \( e \).  Unless \( N \) is factored, it
returns \( c'(x) \), monic modulo \( N \), and \( d'(x) \) such that
\( \langle N^e, c(x), h(x) \rangle
= \langle N^e, c'(x), N d'(x) \rangle \);
we replace \( c(x) \) by \( c'(x) \), and put \( d'(\alpha) \) into
\( I' \).  Finally we put \( N^{e-2} \bar{a}(\alpha) \) into \( I' \),
which is possible since \( e \geq 2 \), and let \( a'(x) \) be the
last \( c(x) \).  The ideal is then \( \langle N^e, a'(\alpha) , N I' \rangle \),
where \( I' \) has \( k+2 \) generators.  By item (a) of
Lemma~\ref{lem:polygcd}, every \( c'(x) \) divides the previous
\( c(x) \) over \( \Z/N \).  Hence \( a'(x) \) divides
\( \tilde{a}(x) \) over \( \Z/N \), and \( \tilde{a}(x) \) divides
\( f(x) \) over \( \Z/N \) by Equation~\eqref{eq:4}.
The reduction modulo \( f(x) \) does not change a polynomial modulo
\( N \) if its degree modulo \( N \) is less than \( n \).

If \( a'(x) \equiv 1 \pmod{N} \), then \( a'(\alpha) \) is a unit
modulo \( N^e \) by Lemma~\ref{lem:unit}, and the ideal is \( O_K \).
Conversely, if the ideal is \( O_K \), then
\( \langle N, a'(\alpha) \rangle = O_K \), namely,
\( \langle N, a'(x), f(x) \rangle = \langle 1 \rangle \) in
\( \Z[x] \).  Since \( f(x) \in \langle N, a'(x) \rangle \), we have
\( \langle N, a'(x) \rangle = \langle 1 \rangle \).  As \( a'(x) \)
is monic modulo \( N \), this implies \( a'(x) \equiv 1 \pmod{N} \),
as in the proof of item (c) of Lemma~\ref{lem:polygcd}.
\end{proof}

We could, of course, reduce the number of generators in \( I' \),
but it may not be worth the computation effort.
We should continue producing ideal factors and put them in
the preferred form.
The following theorem provides a simple algorithm
to multiply them into one ideal in two-element representation.

\begin{theorem}\label{thm:product}
  If \( \tilde{a}_1(x), \tilde{a}_2(x), \cdots, \tilde{a}_k(x) \) all
  satisfy the conditions in Equation~\eqref{eq:4},
  then
  \[ \langle N, \tilde{a}_1(\alpha) \rangle \langle N, \tilde{a}_2 (\alpha) \rangle \cdots \langle N, \tilde{a}_k(\alpha) \rangle
  = \langle N^k, \tilde{a}_1(\alpha)\tilde{a}_2(\alpha)\cdots \tilde{a}_k(\alpha) \rangle\]
\end{theorem}

\begin{proof}
  The two generators of the RHS are products of one generator from
  each factor of the LHS, so the RHS is contained in the LHS.
  The LHS is generated by the products
  \( N^{k-s} \tilde{a}_{i_1}(\alpha) \tilde{a}_{i_2}(\alpha) \cdots
  \tilde{a}_{i_s}(\alpha) \) with \( 0 \leq s \leq k \);
  for \( s = 0 \) and \( s = k \) these are the generators of the RHS.
  Let \( 1\leq s < k \) be an integer,
  and \( i_1, i_2, \cdots, i_s, i_{s+1}, \cdots i_k \)
  be a permutation of \( 1, 2, 3, \cdots, k \).
  We prove the theorem by showing that 
  \( N^{k-s} \tilde{a}_{i_1}(\alpha) \tilde{a}_{i_2}(\alpha) \cdots \tilde{a}_{i_s}(\alpha)\), which is a term in the expansion of the LHS,
  is in the RHS ideal. 
  By applying Lemma~\ref{lem:genN}, we observe that
  for \( s+1 \leq j \leq k \),
  \[ N \in  \langle N^k, \tilde{a}_{i_{ j }}(\alpha) \rangle, \]
  and derive:
  \begin{align*}
     & N^{k-s} \tilde{a}_{i_1}(\alpha) \tilde{a}_{i_2}(\alpha) \cdots \tilde{a}_{i_s}(\alpha)\\
    \in& ( \tilde{a}_{i_1}(\alpha) \tilde{a}_{i_2}(\alpha) \cdots \tilde{a}_{i_s}(\alpha)  )\langle N^k, \tilde{a}_{i_{s+1}}(\alpha) \rangle
          \cdots \langle N^k, \tilde{a}_{i_k}(\alpha) \rangle\\
  \subseteq& \langle N^k, \tilde{a}_1 (\alpha)\tilde{a}_2(\alpha)\cdots \tilde{a}_k(\alpha) \rangle
  \end{align*}
\end{proof}

In the terminology of \cite[Section 6.3]{PZ97} and
\cite[Definition 6.1]{FHHJ17}, Lemma~\ref{lem:genN} shows that
\( \langle N, \tilde{a}(\alpha) \rangle \) is a normal presentation
with respect to the set of primes dividing \( N \).
Indeed, let \( Q \) be a prime ideal of \( O_K \) containing \( N \),
and write \( v_Q \) for the exponent of \( Q \) in the factorization of
an ideal.  Since \( N \in \langle N^2, \tilde{a}(\alpha) \rangle \), we
have \( v_Q(N) \geq \min( 2 v_Q(N), v_Q(\tilde{a}(\alpha)) ) \), hence
\( v_Q(\tilde{a}(\alpha)) \leq v_Q(N) \), and \( \tilde{a}(\alpha) \)
has the same exponent at \( Q \) as the ideal
\( \langle N, \tilde{a}(\alpha) \rangle \), which is the condition of
the definition.
Theorem~\ref{thm:product} is the multiplication rule for normal
presentations \cite[Theorem 6.2]{FHHJ17}; we include its short proof
for completeness.  What is new here is that the condition
\( \langle \tilde{a}(x), f_1(x), N \rangle = \langle 1 \rangle \) of
Equation~\eqref{eq:4} is sufficient for normality, and that it can be
reached deterministically by Theorem~\ref{thm:tech}.
Note that we do not require \( \tilde{a}_i(x) \) to be
relatively prime to each other over \( \Z/N \).
This theorem provides the mathematical foundation for stage 2
of our algorithm. For pseudo-code, see Algorithm~\ref{alg:stage2}.

\begin{algorithm}[h]
  \caption{Sketch of the algorithm in the second stage}
  \label{alg:stage2}
    \KwIn{ The stacks R1 and R2 from the first stage.}
    \KwOut{ Updated R1}
    \While{R2 is not empty}{
      Pop an ideal \( J_0 = \langle N^{e}, a_0(\alpha), N I_0 \rangle \) from R2\;
    \Try{}{
       \( J \gets J_0 \), \( a \gets a_0 \), \( I \gets I_0 \)\;
       \( \Int \gets 1 \)\;
       \( \Poly \gets 1 \)\;
    \While{\( a(x) \not\equiv 1 \pmod{N} \)}
      {
        Find \( \tilde{a}(x) \) for \( a(x) \) as in Theorem~\ref{thm:tech}\;
        \( \Int \gets \Int \times N \)\;
        \( \Poly \gets \Poly \times \tilde{a}(\alpha) \)\;
        \( J \gets J/ \langle N, \tilde{a}(\alpha) \rangle \)
         (Lemma~\ref{lem:other})\;
         Normalize \( J \) and update
         \( a,  I \) so that
         \( J = \langle N^{e}, a(\alpha) , N  I \rangle \)\;
    }
    Push the ideal \( \langle \Int, \Poly \rangle \) into R1\;
    }
    \Catch{\( N \) is factored as \( N = N_1 N_2 \) with
      \( \gcd(N_1, N_2) = 1 \)}
    { \tcc{\( \Int \), \( \Poly \) and the current \( J \) are
        discarded; the two co-prime factors of \( J_0 \) are pushed}
      \For{\( i = 1, 2 \)}{
          Write \( N_i = P_i^{k_i} \) so that \( P_i \) is not a
          perfect power\;
          Push \( \langle P_i^{e k_i}, a_0(\alpha),
          P_i  (P_i^{k_i - 1} I_0 ) \rangle \)
          into the stack R2\;
    \tcc{\( N/N_i \) is a unit modulo \( N_i^e \), so this is
      \( \langle N_i^e, a_0(\alpha), N I_0 \rangle \), and it is in the right form.}
     }
    }
    }
\end{algorithm}

\section{Putting the algorithm together}

We are now ready to prove the main theorem.
The building block of our algorithm is the Euclidean
algorithm for integral polynomials over \( \Z/N \),
where \( N \) is a factor of \( {\cal N} \) and
is not a perfect power.  To analyze the complexity, we count
the number of runs of this algorithm.
We first fix a few details that are left open in the
pseudo-code.
\begin{enumerate}
\item An element of \( O_K \) is represented by a polynomial of degree
  less than \( n \).  When we work with an ideal whose integer
  generator is \( N^e \), every result is reduced modulo \( f(x) \)
  and modulo \( N^{e+2} \).  This precision is sufficient, since the
  elements \( \bar{a}(\alpha) \) and \( g(\alpha) \) in
  Lemma~\ref{lem:other} are needed modulo \( N^{e+1} \), and they are
  obtained after one division by \( N \).
  In particular the polynomial \( A(x) \) in Equation~\eqref{eq:1} is
  computed as the quotient of \( f(x) \) by the monic polynomial
  \( a_1(x) \) over \( \Z/N^{e+2} \).
\item The integer factors of the input ideal are collected in
  \( N_0 \) during the first stage, and they do not appear again.
  Indeed, let \( J = \langle N^e, a(\alpha), N I \rangle \) be an ideal pushed into
  R2 by Algorithm~\ref{alg:stage1}.  Since \( a(x) \) is monic modulo
  \( N \) of degree less than \( n \), the element \( a(\alpha) \) is
  not in \( p O_K \) for any prime \( p \mid N \), so \( J \) is not
  contained in \( p O_K \).  Every ideal met in the second stage is
  either a quotient of such a \( J \) by an integral ideal, or one of
  its co-prime factors; in both cases it contains \( J \), so it is
  not contained in \( p O_K \) either.  Hence its polynomial
  generators are never all multiples of \( N \), and the polynomial
  \( a(x) \) in Algorithm~\ref{alg:stage2} always has degree less than
  \( n \) modulo \( N \).
\item When \( N \) is factored in Algorithm~\ref{alg:stage2},
  we discard \( \Int \) and \( \Poly \), and
  the ideals pushed into R2 are the factors of the ideal \( J \) as
  it was popped from R2.
\item The polynomials \( a'(x), f_1'(x) \) and \( t(x) \) in the proof
  of Lemma~\ref{lem:genN} are found by one run of the algorithm from
  Lemma~\ref{lem:polygcd} on \( \tilde{a}(x) \), \( f_1(x) \) and
  \( N \), which returns \( u(x), v(x) \) with
  \( u(x) \tilde{a}(x) + v(x) f_1(x) \equiv 1 \pmod{N} \).
  The inverse of \( 1 - N t(\alpha) \) modulo \( N^{e+1} \) is found
  by the Newton iteration
  \( y \leftarrow y ( 2 - (1 - N t(\alpha)) y ) \) starting from
  \( y = 1 \), which takes \( O(\log e) \) multiplications in
  \( O_K \).
\end{enumerate}

\begin{proof}[of Theorem~\ref{thm:main}]
  We run Algorithm~\ref{alg:stage1},
  Algorithm~\ref{alg:stage2} and Algorithm~\ref{alg:stage3}
  in sequence.

  {\bf Correctness.}
  Whenever an integer is split into co-prime factors, the ideal is
  split accordingly by Lemma~\ref{lem:factorideal}; otherwise
  Algorithm~\ref{alg:stage1} takes the integer factor \( N^k \) out of
  an ideal, and changes the generators of the other factor
  by Lemma~\ref{lem:polygcd}.  Hence after the first stage the product
  of \( N_0 \) and the ideals in R1 and R2 is \( {\cal I} \), the integer
  generators of these ideals are pairwise co-prime, and \( N_0 \)
  times their product is \( {\cal N} \).
  Let \( J = \langle N^e, a(\alpha), N I \rangle \) be an ideal popped from R2, and
  suppose that \( N \) is not factored while \( J \) is processed.
  Let \( \tilde{a}_1(x), \tilde{a}_2(x), \cdots, \tilde{a}_L(x) \) be
  the polynomials found in the inner loop.  By
  Lemma~\ref{lem:other}, after the \( j \)-th pass the ideal is
  \( J / \prod_{i \leq j} \langle N, \tilde{a}_i(\alpha) \rangle \).  The loop
  stops when \( a(x) \equiv 1 \pmod{N} \), which happens exactly when
  the ideal is \( O_K \), by Lemma~\ref{lem:other}; the last
  paragraph of its proof also applies to the ideal as it is popped
  from R2.  Therefore, by
  Theorem~\ref{thm:product},
  \[ J = \langle N, \tilde{a}_1(\alpha) \rangle \langle N, \tilde{a}_2 (\alpha) \rangle \cdots
    \langle N, \tilde{a}_L(\alpha) \rangle
    = \langle N^L, \tilde{a}_1(\alpha)\tilde{a}_2(\alpha)\cdots
    \tilde{a}_L(\alpha) \rangle, \]
  which is the ideal \( \langle \Int, \Poly \rangle \) pushed into R1.  We will see
  below that \( L \leq e \), so the integer generators in R1 remain
  pairwise co-prime, and \( N_0 \) times their product divides
  \( {\cal N} \).  The third stage merges the ideals in R1 into one
  ideal \( \langle M', c'(\alpha) \rangle \), which is correct by
  Lemma~\ref{lem:relativeprime}, and returns \( M = N_0 M' \) and
  \( c(x) = N_0 c'(x) \), which is correct since
  \( N_0 \langle M', c'(\alpha) \rangle = \langle N_0 M', N_0 c'(\alpha) \rangle \).

  {\bf Unit of cost.}
  We call a \emph{basic step} one run of the algorithm from
  Lemma~\ref{lem:polygcd}, or one multiplication or exact division of
  two polynomials of degree at most \( 2n \) followed by a reduction
  modulo \( f(x) \).  For an ideal with integer generator \( N^e \),
  a basic step takes \( O(n^2) \) operations in \( \Z/N^{e+2} \),
  hence time \( \tilde{O}(n^2 \log N^e) \).

  {\bf Number of pops.}
  Consider the integers popped from S during the first stage.  A pop
  either succeeds, or it replaces the integer by at least two
  pairwise co-prime factors greater than one.  So the pops form a
  tree whose leaves are pairwise co-prime divisors of
  \( {\cal N} \) greater than one.
  There are at most \( \omega({\cal N}) \) leaves, hence at most
  \( 2\omega({\cal N}) - 1 \) pops, and the depth of the tree is less
  than \( \omega({\cal N}) \).  The integers at the same depth are
  pairwise co-prime and their product divides \( {\cal N} \).
  Consequently, if a pop with integer \( N^e \) takes
  time \( \tilde{O}(C n^2 \log N^e) \), where \( C \) does not
  depend on the pop, then all the pops together take time
  \( \tilde{O}(C n^2 \omega({\cal N}) \log {\cal N}) \).
  The same holds for the ideals popped from R2 in the second stage,
  because of the convention (3) above.

  {\bf Stage 1.}
  A pop uses at most two basic steps, so \( C = O(1) \).
  The exponent \( k \) is found from the greatest common divisor of
  \( N^e \) and the coefficients of \( a(x) \) and \( b(x) \), in
  time \( \tilde{O}(n \log N^e) \).  In addition
  we need to write the popped integer \( P \) as \( N^e \), where
  \( N \) is not a perfect power.  The algorithms in
  \cite{Bernstein98,BLP07} decide in time \( \tilde{O}(\log P) \)
  whether \( P \) is a perfect power, and if it is, write it as
  \( y^j \) with \( j \geq 2 \).  We apply them to \( P \), then to
  \( y \), and so on; since the bit length is at least halved every
  time, at most \( \log_2\log_2 P \) rounds are needed.  When a zero
  divisor is found, the factor refinement in the proof of
  Lemma~\ref{lem:polygcd} takes time \( \tilde{O}(\log N) \) by
  \cite{Bernstein05}.  The first stage thus takes time
  \( \tilde{O}(n^2 \omega({\cal N}) \log {\cal N}) \).
  If \( {\cal N} \) has no repeated prime factor, then \( e = 1 \) for
  every pop and R2 is empty after the first stage.  Together
  with the cost of the third stage below, this proves the second
  bound of the theorem.

  {\bf Stage 2.}
  Let \( J = \langle N^e, a(\alpha), N I \rangle \) with \( e \geq 2 \) be an ideal
  popped from R2.  We claim that:
  \begin{enumerate}
  \item[(a)] the inner loop makes \( L \leq \min(e, m) \) passes;
  \item[(b)] in every pass, finding \( \tilde{a}(x) \) and
    \( \bar{a}(\alpha) \) takes \( O(m + \log e) \) basic steps;
  \item[(c)] the \( j \)-th pass takes \( O(j) \) further basic steps
    in Lemma~\ref{lem:other}.
  \end{enumerate}
  Together with the \( L - 1 \) multiplications for \( \Poly \), a pop
  then takes
  \( O( \sum_{j \leq L} (j + m + \log e) ) = O(m^2 + m \log e) \)
  basic steps, so \( C = m^2 + m\log e \), and the second stage
  takes time
  \( \tilde{O}( n^2 \omega({\cal N}) m^2 \log {\cal N} ) \).

  For a prime ideal \( Q \), denote by \( v_Q \) the exponent of
  \( Q \) in the factorization of an ideal.
  To see (a), recall that
  \( \langle N, \tilde{a}(\alpha) \rangle = \langle N, a(\alpha) \rangle = J + N O_K \), so
  \( v_Q(\langle N, \tilde{a}(\alpha) \rangle) = \min( v_Q(J), v_Q(N) ) \), and one
  pass changes \( v_Q(J) \) into
  \( \max(0, v_Q(J) - v_Q(N)) \).  At the beginning
  \( v_Q(J) \leq e\, v_Q(N) \), since \( N^e \in J \), and
  \( v_Q(J) \leq v_Q({\cal I}) \leq m \), since \( J \) is a factor of
  \( {\cal I} \).  Moreover \( v_Q(N) \geq 1 \) whenever \( v_Q(J) \geq 1 \).
  So after \( \min(e,m) \) passes we have \( J = O_K \), and the
  loop stops.

  To see (b), let \( a_1(x) \) be the monic polynomial with
  \( a_1(x) \equiv a(x) \pmod{N} \), and let \( p \) be a prime factor
  of \( N \).  The prime ideals containing \( p \) are
  \( Q = \langle p, \varphi(\alpha) \rangle \), where \( \varphi(x) \) runs over the
  irreducible factors of \( f(x) \) over \( \F_p \) \cite{Cohen93},
  and the multiplicity of \( \varphi(x) \) in \( a_1(x) \) over
  \( \F_p \) is the exponent of \( Q \) in
  \( \langle p, a_1(\alpha) \rangle = J + p O_K \), which is at most
  \( v_Q(J) \leq m \).  In the proof of Lemma~\ref{lem:split}, if
  \( g_j(x) \neq 1 \), then \( g_j(x) \) has positive degree modulo
  \( p \), and passing from \( h_j(x) \) to \( h_{j+1}(x) \) lowers
  the multiplicity of every irreducible factor of \( g_j(x) \) over
  \( \F_p \); such a factor divides \( f_1(x) \) over \( \F_p \), so
  it divides all the following \( g_i(x) \) until it disappears
  from \( h_i(x) \).  Hence \( g_j(x) \neq 1 \) happens at
  most \( m \) times, and Lemma~\ref{lem:split} uses at most
  \( m + 2 \) runs of the algorithm from Lemma~\ref{lem:polygcd} and
  \( O(m) \) other basic steps.  By the convention (4),
  \( \bar{a}(\alpha) \) is then found in \( O(\log e) \) basic steps.

  To see (c), note that \( I \) has two generators when \( J \) is
  pushed into R2, and at most two more after every pass by
  Lemma~\ref{lem:other}.  So in the \( j \)-th pass \( I \) has at
  most \( 2j \) generators, and the proof of Lemma~\ref{lem:other}
  uses at most \( 2j + 2 \) multiplications by \( \bar{a}(\alpha) \)
  and at most \( 2j + 1 \) runs of the algorithm from
  Lemma~\ref{lem:polygcd}.

  {\bf Stage 3.}
  There are at most \( \omega({\cal N}) \) ideals in R1.  Each
  iteration of Algorithm~\ref{alg:stage3} merges two of them, using
  two modular inverses and \( O(n) \) multiplications modulo an
  integer dividing \( {\cal N} \).  The third stage takes time
  \( \tilde{O}(n\, \omega({\cal N}) \log {\cal N}) \).
\end{proof}

The proof gives slightly more.  Let \( \rho \) be the largest
ramification index of a prime ideal containing \( {\cal N} \), and let
\( e \) be the largest exponent met in the second stage.  The number
of passes is at most \( \tau = \min(e, m) \), and the number \( m \)
in (b) can be replaced by \( r = \min(\rho, m) \), so the factor
\( m^2 \) in Theorem~\ref{thm:main} can be replaced by
\( 1 + \tau(\tau + r) \).  In particular \( r = 1 \) if \( {\cal N} \)
is co-prime to the discriminant of \( f(x) \).
The term \( \tau^2 \) comes from the number of
generators of \( I' \) in Lemma~\ref{lem:other}; it would become
\( \tau \) if \( I' \) were kept with a bounded number of generators.
The Euclidean algorithm in Lemma~\ref{lem:polygcd} is the quadratic
one.  Asymptotically fast algorithms over \( \Z/N \), with the same
treatment of zero divisors, are studied in \cite{FHS20}; we have not
analyzed their use here.

\begin{remark}\label{rem:nonmono}
  Theorem~\ref{thm:main} remains true for a monic irreducible
  polynomial \( f(x) \) that is not monogenic, provided that
  \( {\cal N} \) is co-prime to the index \( [O_K:\Z[\alpha]] \).
  Indeed, multiplication by \( {\cal N}^s \) is then a bijection on
  the finite group \( O_K/\Z[\alpha] \), so the inclusion induces an
  isomorphism from \( \Z[\alpha]/{\cal N}^s \Z[\alpha] \) to
  \( O_K/{\cal N}^s O_K \) for every \( s \geq 1 \), and the ideals
  of \( O_K \) containing \( {\cal N}^s \) correspond to the ideals of
  \( \Z[\alpha] \) containing \( {\cal N}^s \).  All the ideals in our
  algorithm contain a power of \( {\cal N} \), and all the
  computations are done in \( \Z[x] \) modulo \( f(x) \) and a power
  of a divisor of \( {\cal N} \).  The assumption that \( f(x) \) is
  monogenic is used in three places only.  The first is
  Lemma~\ref{lem:split}, where Lemma~\ref{lem:gended} gives the
  contradiction as soon as \( N > 1 \) is co-prime to the index.
  The second is the description of the prime ideals containing a
  prime \( p \mid {\cal N} \) in the proof of
  Theorem~\ref{thm:main}, which holds whenever \( p \) does not
  divide the index \cite{Cohen93}.  The third is the fact that
  \( a(\alpha) \in p O_K \) implies \( a(x) \equiv 0 \pmod{p} \) for a
  polynomial \( a(x) \) of degree less than \( n \), which holds
  since \( p O_K \cap \Z[\alpha] = p \Z[\alpha] \) for such \( p \).

  An integer is co-prime to the index if and only if it is co-prime
  to the conductor of \( \Z[\alpha] \) in \( O_K \), because a prime
  \( p \) divides the index if and only if some prime ideal above
  \( p \) contains the conductor.  We do not treat the weaker
  assumption that the ideal \( {\cal I} \) itself is co-prime to the
  conductor.

  If nothing is known about the index, the algorithm can still be
  run.  Whenever the contradiction in the proof of
  Lemma~\ref{lem:split} occurs, the proof of Lemma~\ref{lem:gended}
  gives an element of \( O_K \) that is not in \( \Z[\alpha] \), which
  is a certificate that \( f(x) \) is not monogenic.
\end{remark}

Theorem~\ref{thm:main} gives the arithmetic of ideals in two-element
representation.

\begin{corollary}[Ideal arithmetic]\label{cor:arith}
  Let \( f(x) \) be a monogenic polynomial of degree \( n \), and let
  \( \fa = \langle M_1, c_1(\alpha) \rangle \) and
  \( \mathfrak{b} = \langle M_2, c_2(\alpha) \rangle \) be two ideals of \( O_K \),
  where \( M_1, M_2 > 1 \) are integers and
  \( c_1(x), c_2(x) \in \Z[x] \) have degree less than \( n \).
  There are deterministic algorithms, with running time polynomial in
  \( n \) and \( \log (M_1 M_2) \), that find the following.
  \begin{enumerate}
  \item A two-element representation of \( \fa + \mathfrak{b} \).
  \item A two-element representation of \( \fa \mathfrak{b} \).
  \item An integer \( M \) dividing \( M_2 \), and a two-element
    representation \( \langle M', \bar{c}(\alpha) \rangle \) of an ideal of
    \( O_K \), such that
    \( \mathfrak{b}^{-1} = \frac{1}{M} \langle M', \bar{c}(\alpha) \rangle \).
    Consequently the quotient
    \( \fa \mathfrak{b}^{-1}
    = \frac{1}{M}\, \fa\, \langle M', \bar{c}(\alpha) \rangle \)
    is found by one multiplication.
  \end{enumerate}
\end{corollary}

\begin{proof}
  (1) We have
  \( \fa + \mathfrak{b}
  = \langle \gcd(M_1, M_2), c_1(\alpha), c_2(\alpha) \rangle \).
  If \( \gcd(M_1, M_2) = 1 \), this is \( O_K \).  Otherwise we apply
  Theorem~\ref{thm:main} with \( {\cal N} = \gcd(M_1, M_2) \), after
  reducing \( c_1(x) \) and \( c_2(x) \) modulo \( {\cal N} \).

  (2) We have
  \( \fa \mathfrak{b} = \langle M_1 M_2, M_1 c_2(\alpha), M_2 c_1(\alpha),
  c_1(\alpha) c_2(\alpha) \rangle \).
  Theorem~\ref{thm:main} with \( {\cal N} = M_1 M_2 \) gives
  \( \langle M, c(\alpha) \rangle
  = \langle M_1 M_2, M_1 c_2(\alpha), M_2 c_1(\alpha) \rangle \),
  and a second application, with \( {\cal N} = M \), gives a
  two-element representation of
  \( \langle M, c(\alpha), c_1(\alpha) c_2(\alpha) \rangle = \fa \mathfrak{b} \).
  If \( M = 1 \), then \( \fa \mathfrak{b} = O_K \).

  (3) We run the first two stages of our algorithm on the ideal
  \( \mathfrak{b} = \langle M_2, c_2(\alpha), 0 \rangle \).
  For an ideal \( \langle N, c(\alpha) \rangle \) pushed into R1 by the first
  stage, we replace \( c(x) \) by the greatest common divisor
  \( a_1(x) \) of \( c(x) \) and \( f(x) \) over \( \Z/N \), which
  does not change the ideal, and, unless
  \( a_1(x) \equiv 1 \pmod{N} \), we find \( \tilde{a}(x) \) for
  \( a_1(x) \) as in Theorem~\ref{thm:tech}.  If \( N \) is factored
  in this step, we go back to the first stage with the two factors.
  Together with the factors found in the second stage, this writes
  \[ \mathfrak{b} = N_0 \prod_i \prod_{j=1}^{L_i}
    \langle N_i, \tilde{a}_{ij}(\alpha) \rangle, \]
  where the \( N_i \) are pairwise co-prime, and every
  \( \tilde{a}_{ij}(x) \) satisfies the conditions in
  Equation~\eqref{eq:4} with respect to \( N_i \).
  Put \( M' = \prod_i N_i^{L_i} \) and \( M = N_0 M' \), which
  divides \( M_2 \).
  By Corollary~\ref{cor:div} with \( s = L_i + 1 \), we find
  \( \bar{a}_{ij}(\alpha) \) such that
  \( \tilde{a}_{ij}(\alpha) \bar{a}_{ij}(\alpha) \equiv N_i
  \pmod{N_i^{L_i+1}} \) and
  \( \langle N_i, \tilde{a}_{ij}(\alpha) \rangle^{-1}
  = \langle N_i, \bar{a}_{ij}(\alpha) \rangle/N_i \);
  for the factors found in the second stage, these elements were
  already computed in Lemma~\ref{lem:other}.  Hence
  \[ \mathfrak{b}^{-1} = \frac{1}{M} \prod_i \prod_{j=1}^{L_i}
    \langle N_i, \bar{a}_{ij}(\alpha) \rangle. \]
  The congruence shows that
  \( N_i \in \langle N_i^{L_i}, \bar{a}_{ij}(\alpha) \rangle \).  This is the
  conclusion of Lemma~\ref{lem:genN} for \( \bar{a}_{ij} \), and it
  is the only property used in the proof of
  Theorem~\ref{thm:product}.  Therefore
  \( \prod_{j} \langle N_i, \bar{a}_{ij}(\alpha) \rangle
  = \langle N_i^{L_i}, \prod_{j} \bar{a}_{ij}(\alpha) \rangle \) for every \( i \),
  and these ideals are merged into one ideal
  \( \langle M', \bar{c}(\alpha) \rangle \) by Lemma~\ref{lem:relativeprime}.

  Finally, let \( \langle M'', c''(\alpha) \rangle \) be the two-element
  representation of \( \fa\, \langle M', \bar{c}(\alpha) \rangle \) found by (2),
  where \( c''(x) \) has degree less than \( n \).  Then
  \( \fa \mathfrak{b}^{-1} = \frac{1}{M} \langle M'', c''(\alpha) \rangle \), and it
  is an integral ideal if and only if \( M \) divides \( M'' \) and
  all the coefficients of \( c''(x) \).
\end{proof}

\begin{algorithm}[h]
  \caption{Sketch of the algorithm in the third stage}
  \label{alg:stage3}
  \KwIn{The integer \( N_0 \) and the stack R1}
      \KwOut{ two elements \( M \in \Z \) and \( c(x) \in \Z[x] \) }
    \If{R1 is empty}{
      \tcc{the input ideal is \( \langle N_0 \rangle \)}
      \Return \( (N_0, 0) \)\;}
    \While{R1 has at least two elements}{
        Pop the stack to \( \langle M_1, c_1 \rangle \)\;
        Pop the stack to \( \langle M_2, c_2 \rangle \)\;
        \( u_1 \gets M_1^{-1} \pmod{M_2} \)\;
        \( u_2 \gets M_2^{-1} \pmod{M_1} \)\;
        \( c \gets M_2 u_2 c_1 + M_1 u_1 c_2 \pmod{M_1 M_2} \)\;
        Push \( \langle M_1 M_2, c \rangle \) into R1\;
   }
   Pop the stack to \( \langle M_1, c_1 \rangle \)\;
   \tcc{\( N_0 \langle M_1, c_1(\alpha) \rangle = \langle N_0 M_1, N_0 c_1(\alpha) \rangle \)}
   \Return \( (N_0 M_1, N_0 c_1) \)\;
 \end{algorithm}

\section{Experiments}\label{sec:exp}

We implemented Algorithms~\ref{alg:stage1}--\ref{alg:stage3} and the
multiplication of Corollary~\ref{cor:arith} in Python, on top of the
polynomial arithmetic of PARI \cite{PARI2}.
For Lemma~\ref{lem:polygcd}, the prototype calls the polynomial gcd of
PARI over \( \Z/N \), which stops with an error at a leading coefficient
that is not invertible, and lifts the result to \( \Z/N^e \); running
the loop from the proof of the lemma in Python instead is up to five
times slower.
On 19 monogenic fields of degree at most 16, the outputs for 5890
random ideals given by three generators, and for 5890 random products,
agree with PARI; the tests include ramified primes and repeated prime
ideal factors.

Table~\ref{tab:mult} compares the multiplication of two ideals of
\( K = \Q[x]/(x^n+1) \) with the function \texttt{idealmul} of PARI,
which works with Hermite Normal Forms.
Here \( P \) and \( P' \) are prime ideals of residue degree one above
two 64-bit primes, and \( Q \) is a prime ideal of residue degree
\( n/2 \) above a 64-bit prime.  Our input is \( \langle M, c(\alpha) \rangle \),
where \( M \) is the minimal positive integer of the ideal, and
\( c(\alpha) \) is a generic element of the ideal reduced modulo
\( M \).

\begin{table}[h]
  \centering
  \caption{Time of one multiplication in milliseconds:
    our prototype / PARI \texttt{idealmul}.}
  \label{tab:mult}
  \begin{tabular}{l@{\quad}r@{\quad}r@{\quad}r}
    \hline
    product & \( n = 64 \) & \( n = 128 \) & \( n = 256 \) \\
    \hline
    \( P \cdot P' \)   & 0.9 / 4.1  & 2.4 / 25.7  & 10.8 / 192 \\
    \( P \cdot P \)    & 3.2 / 4.1  & 8.3 / 26.0  & 24.7 / 189 \\
    \( P^3 \cdot P^2 \) & 16.3 / 4.6 & 40.8 / 28.1 & 126 / 200 \\
    \( Q \cdot Q \)    & 2.5 / 12.5 & 6.4 / 168   & 18.0 / 5020 \\
    \hline
  \end{tabular}
\end{table}

The advantage grows with \( n \), in agreement with
Theorem~\ref{thm:main}, and our time is not sensitive to the residue degree.
For \( n = 32 \), PARI is faster except in the first row.
The third row shows the term \( m^2 \) of Theorem~\ref{thm:main}:
the second stage makes eight passes there.
If PARI has to start from the same two generators and to return two
generators, which adds two Hermite Normal Form computations and its
randomized function \texttt{idealtwoelt}, then for \( n = 128 \) it
is between 4.6 and 62 times slower than the prototype on these four
products.
The timings are medians over three instances (two for \( n = 256 \))
on a shared machine, and the control flow of the prototype is
interpreted, so they only indicate the order of magnitude.

 \section{Conclusion and open problems}
In this paper, we present a deterministic polynomial time algorithm
that finds a two-element representation of an ideal
for monogenic polynomials.
Cyclotomic polynomials are monogenic,
making our algorithms directly applicable to
the most important case for cryptography based on ideal lattices.
The second stage constructs the normal presentations of
Pohst and Zassenhaus \cite{PZ97} for the intermediate factors
with a common integer generator of unknown factorization.
The implementation in Section~\ref{sec:exp} is a prototype.
An implementation in a compiled language, which covers all the
operations of Corollary~\ref{cor:arith} and reuses the normal
presentations of intermediate factors in successive multiplications,
is left to future work.

\bibliographystyle{splncs04}
\bibliography{crypto}

\end{document}